\documentclass[12pt]{amsart}

\usepackage{graphicx, fullpage}
\usepackage{amsmath,amssymb,amsthm,amscd, bm}
\usepackage{fancyhdr}
\usepackage[mathscr]{eucal}
\usepackage{amsfonts}
\usepackage[T1]{fontenc}
\usepackage{xspace}
\usepackage{esint}
\theoremstyle{plain}
\newtheorem {thm}{Theorem}[section]

\newtheorem {prop}[thm]{Proposition}
\numberwithin{equation}{section}

\def\cal{\mathcal}

\newcommand{\cF}{\cal F}

\newcommand{\cE}{{\cal E}}

\newcommand{\cP}{{\cal P}}

\newcommand{\bC}{\mathbb C}
\newcommand{\bE}{\mathbb E}
\newcommand{\bP}{\mathbb P}

\newcommand{\bL}{\mathbb L}
\newcommand{\bR}{\mathbb R}

\newcommand{\tr}{\text{tr}\,}

\newcommand{\Id}{\text{Id}}
\newcommand{\GL}{\text{GL}}

\newcommand{\U}{\text{U}}

\newcommand{\re}{\text{Re}\,}
\newcommand{\im}{\text{Im}\,}

\newcommand{\vo}{\text{Vol}}

\begin{document}
\title{\bf Ellipsoids in pseudoconvex domains}
\author{L\'aszl\'o Lempert}
\address{Department of  Mathematics,
Purdue University, 150N University Street, West Lafayette, IN
47907-2067, USA}
\subjclass[2020]{32T, 32U05, 52A40}

\abstract
We consider the problem of maximizing the volume of hermitian ellipsoids inscribed in a given pseudoconvex domain in 
complex Euclidean space. We prove existence and uniqueness, and give a characterization of the maximizer. 
 
\endabstract
\maketitle
\section{Introduction}    

Consider a pseudoconvex domain $\Omega\subset\bC^n$ (\cite[Definition 2.6.8]{H91}) and the hermitian 
ellipsoids---images of the unit ball under invertible complex linear transformations---that it contains. 
How to find among these ellipsoids the one(s) 
of maximal volume? This is a complex analog of a convex geometry problem (whose dual) John posed and essentially
solved in  \cite{J48}, see also \cite{B97}. In that problem one is to find the ellipsoid of maximal volume inscribed in a given
bounded convex domain $\Omega\subset\bR^n$. One complex analog still considers bounded convex domains
$\Omega$, but this time in $\bC^n$, and searches for the inscribed hermitian ellipsoid of maximal volume. 
The arguments of Arocha, 
Bracho, and Montejano in \cite{A+22} give that this maximal hermitian ellipsoid exists and is unique. 
Earlier related work is by Gromov, Rabotin, and Ma \cite{G67, M97, R86}.

However, in complex geometry what oftentimes corresponds to convexity is not ordinary convexity, but the more
general notion of pseudoconvexity. It therefore makes sense to look into the problem of maximal hermitian ellipsoids
inscribed in a bounded pseudoconvex domain $\Omega\subset\bC^n$. We emphasize that  our definition of a 
hermitian ellipsoid implies that its center is at $0$, and maximizing volume over inscribed translates of hermitian ellipsoids 
is a problem of a different nature.
We discuss that problem in section 5, to the extent it deserves.

Denote by $\bE_0$ the space of hermitian ellipsoids in $\bC^n$. This space is in bijective correspondence with
the space of positive definite hermitian forms $h:\bC^n\times\bC^n\to\bC$, the ellipsoid corresponding to
$h$ being
\begin{equation} 
E_h=\{z\in\bC^n:h(z,z)<1\}.
\end{equation} 
\begin{thm} 
Suppose $\Omega\subset\bC^n$ is a bounded pseudoconvex domain, $0\in\Omega$. Among $E\in\bE_0$ 
contained in $\Omega$ there is at least one that has maximal volume. $E_h$ maximizes volume if and only 
if there is a Borel measure $\mu$ on $\partial\Omega\cap\partial E_h$ such that for every linear operator 
$T:\bC^n\to\bC^n$
\begin{equation} 
\int_{\partial\Omega\cap\partial E_h}h(Tz,z)\,d\mu(z)=\tr T.
\end{equation}
\end{thm}

An ellipsoid $E\in\bE_0$ that has maximal volume among those inscribed in $\Omega$ will simply be called
maximal ellipsoid (in $\Omega$).

With $\Delta\subset\bC$ a disc, we will call the range of a nonconstant holomorphic map $\Delta\to\bC^n$ a 
holomorphic disc.

\begin{thm} 
The maximal ellipsoid in Theorem 1.1 is not necessarily unique; but it is unique if $\partial\Omega$ contains
no holomorphic disc.
\end{thm}

For example, the maximal ellipsoid is unique if $\Omega$ is strongly pseudoconvex.

The proofs depend on a property of geodesics in $\bE_0$ (a symmetric space $\approx\GL(n,\bC)/\U(n)$).
Given $E_h,E_k\in\bE_0$, there is a unique positive self adjoint operator $A$ on $(\bC^n,h)$ such that
$AE_h=E_k$; the geodesic $\cE:[0,1]\to\bE_0$ between $E_h,E_k$ is then given by
\[
\cE(t)=A^tE_h\in\bE_0,\qquad 0\le t\le1.
\]
Equivalently, if hermitian forms $h_t$ are defined by
\[
h_t(z,w)=h(A^{-t}z,A^{-t}w),
\]
then $\cE(t)=E_{h_t}$. The key to both Theorems 1.1, 1.2 is
\begin{thm} 
Suppose $\Omega\subset\bC^n$ is a pseudoconvex domain and $\cE:[0,1]\to\bE_0$ is a geodesic. If the
endpoints $\cE(0),\cE(1)$ are inscribed in $\Omega$, then so are all $\cE(t)$, $0\le t\le 1$.
\end{thm}

In other words, the ellipsoids inscribed in $\Omega$ form a convex subset of $\bE_0$.

In \cite{L24} we dealt with another complex analog of John's problem, involving K\"ahler metrics. There is
an overlap between the two works: An $\Omega\subset\bC^n$ is balanced if 
$\lambda\Omega\subset\Omega$ whenever
$\lambda\in\bC$ has modulus $\le 1$. A balanced, strongly pseudoconvex $\Omega$
induces a K\"ahler metric on $\bP_{n-1}$, and for such $\Omega$
Theorems 1.1 and 1.2 can be deduced from \cite[Theorem 1.3]{L24}.

\section{Geodesics in the space of ellipsoids} 

In this section we will prove Theorem 1.3. For the moment,
consider a geodesic $\cE:[0,1]\to\bE_0$. Let $\cE(t)=E_{h_t}$ and $A$ a positive self adjoint operator on 
$(\bC^n,h_0)$ such that $A\cE(0)=\cE(1)$.

\begin{prop} 
$A$ is self adjoint on $(\bC^n,h_t)$, $0\le t\le 1$.
\end{prop}
\begin{proof}
Since $A^{-t}$ and $A$ commute,
\[
h_t(Az,w)=h_0(A^{-t}Az,A^{-t}w)=h_0(A^{-t}z,AA^{-t}w)=h_t(z,Aw),\qquad \text{q.e.d.}
\]
\end{proof}

Now let $\Omega\subset\bC^n$ be pseudoconvex.

\begin{prop} 
Suppose that $\cE(t)\subset\Omega$ for $t\in[0,1]$. If the closure $\overline{\cE(t)}$ is contained in $\Omega$
for $t=0,1$, then the same holds for all $t\in[0,1]$.
\end{prop}
\begin{proof}
Let $S$ denote the strip $\{s\in\bC:0<\re s<1\}$. Fix a continuous plurisubharmonic 
exhaustion function\footnote{That is, $\{\zeta\in\Omega:u(\zeta)\le c\}$ is compact for all $c\in\bR$.}
$u:\Omega\to\bR$ and let
\[
M=\sup\{u(\zeta):\zeta\in\cE(0)\cup\cE(1)\}<\infty.
\]
Choose an arbitrary $z\in\cE(0)$ and put $\lambda=\sqrt{h_0(z,z)}<1$. Since $A^{i\tau}$ is unitary on
$(\bC^n,h_0)$ when $\tau\in\bR$,
\[
A^sz=A^{\re s}A^{i\im s}z\in A^{\re s}\big(\lambda\overline{\cE(0)}\big)\subset\Omega,\qquad s\in\overline S.
\]
The range of the map $[0,1]\times\lambda\overline{\cE(0)}\ni(t,w)\mapsto A^tw\in\Omega$ being compact,
$v(s)=u(A^sz)$ defines a bounded continuous function of $s\in\overline S$, subharmonic on $S$. 
By the maximum
principle, for $0\le t\le 1$
\[
v(t)\le\sup\{v(s):s\in\partial S\}\le\sup\{u(\zeta):\zeta\in\cE(0)\cup\cE(1)\}= M,
\]
i.e., $u(A^tz)\le M$. Varying $z\in\cE(0)$ we obtain $\cE(t)=A^t\cE(0)\subset\{\zeta\in\Omega:u(\zeta)\le M\}$. This
latter being a compact subset of $\Omega$, $\overline{\cE(t)}\subset\Omega$ indeed follows.
\end{proof}
\begin{proof}[Proof of Theorem 1.3]
Now all we are given is that $\cE:[0,1]\to\bE_0$ is a geodesic with endpoints inscribed in $\Omega$. Assume first,
nonetheless, that even $\overline{\cE(0)},\overline{\cE(1)}\subset\Omega$, and prove $\cE(t)\subset\Omega$ 
for all $t$ in this case. Connect $\cE(0),\cE(1)$ with a continuous path $\cP:[0,1]\to\bE_0$ such that 
$\overline{\cP(\tau)}\subset\Omega$ for all $\tau\in[0,1]$, and for each $\tau\in[0,1]$ let 
$\cF(\tau,\cdot):[0,1]\to\bE_0$ denote the geodesic between $\cE(0)=\cP(0)$ and $\cP(\tau)$. In particular,
$\cF(0,\cdot)\equiv\cE(0)$ and $\cF(1,\cdot)=\cE$. Note that $\cF:[0,1]\times[0,1]\to\bE_0$ is continuous. Let
\[
\Theta=\{\tau\in[0,1]:\cF(\tau,t)\subset\Omega\text{ for all } t\in[0,1]\}.
\]
Clearly $\Theta$ is closed and $0\in \Theta$. But $\Theta$ is open in $[0,1]$ as well, for if $\tau\in \Theta$ 
then Proposition 2.2
implies that each closure $\overline{\cF(\tau,t)}\subset\Omega$ is at positive distance to $\partial\Omega$,
$0\le t\le 1$; hence the same holds for $\overline{\cF(\tau',t)}$ if $\tau'\in[0,1]$ is close to $\tau$. Therefore 
$\Theta=[0,1]$ and $\cE(t)=\cF(1,t)\subset\Omega$.

This takes care of the case when $\overline{\cE(0)},\overline{\cE(1)}\subset\Omega$. The general case is
obtained by using what we have just proved for $\lambda\cE(t)$ instead of $\cE(t)$, 
$0<\lambda<1$, since $\cE(t)=\bigcup_{0<\lambda<1}\lambda\cE(t)$.

\end{proof}

\section{The proof of Theorem 1.1} 

That among hermitian ellipsoids inscribed in $\Omega$ there is a maximal one is proved no differently than 
in \cite{J48}. If $D\in\bE_0$ is fixed, any $E\in\bE_0$ is of form $AD$ with $A$ 
a linear operator on $\bC^n$; the condition
that $AD\subset\Omega$ cuts out a compact set of operators $A$, over which $\vo(AD)=|\det A|\vo D$
attains its maximum at an invertible $A_0$; then $A_0D\in\bE_0$ is a maximal ellipsoid sought.

Next suppose that $E_h\in\bE_0$ is maximal in $\Omega$. Consider $L(\bC^n)$, the space of
$\bC$--linear operators on $\bC^n$, as a  real vector space, its dual $L(\bC^n)'$, and linear forms 
$l_z\in L(\bC^n)'$ ($z\in\bC^n$) and $l\in L(\bC^n)'$,
\[
l_z(T)=\re h(Tz,z),\qquad l(T)=\re\tr T/n,\quad\qquad T\in L(\bC^n).
\]
We claim that $l$ is in the convex hull of the set
\begin{equation} 
\{l_z:z\in\partial\Omega\cap\partial E_h\}\subset L(\bC^n)';
\end{equation}
in other words, the value that any linear form on $L(\bC^n)'$ takes at $l$ is dominated by the maximum of the
form on the
set in (3.1). As linear forms on $L(\bC^n)'$ are evaluations at some $T\in L(\bC^n)$, we need to show
\begin{equation} 
\re \tr T/n\le \max\{\re h(Tz,z): z\in\partial\Omega\cap\partial E_h\}\qquad\text{for all} \quad T\in L(\bC^n).
\end{equation}

Given $T$, denote the maximum in (3.2) by $M$. Let $\lambda>M$ and $S=\lambda\Id-T$. If $z\in\partial\Omega\cap\partial E_h$ then
$\re h(Sz,z)>0$. Since $\partial\Omega\cap\partial E_h$ is compact, there are a $\delta>0$ and a neighborhood
$N\subset\bC^n$ of $\partial\Omega\cap\partial E_h$ such that $\re h(Sz,z)>\delta$ when 
$z\in N$. For any $t\in\bR$
let
\[
\cE(t)=e^{-tS}E_h=\{z\in\bC^n: h(e^{tS}z,e^{tS}z)<1\}\in\bE_0.
\]
When $t>0$ is small and $z\in N$
\[
h(e^{tS}z,e^{tS}z)=h(z,z)+2t\,\re h(Sz,z)+O(t^2)>h(z,z),
\]
so $\overline{\cE(t)}\cap N\subset E_h\subset\Omega$. But 
$\overline {\cE(0)}\setminus N=\overline{E_h}\setminus N$ is at positive distance to
$\bC^n\setminus\Omega$, hence $\overline {\cE(t)}\setminus N\subset \Omega$; in sum, 
$\cE(t)\subset\Omega$ for small $t>0$. Since $E_h$ was maximal,
\[
\vo E_h\ge\vo \cE(t)=|\det e^{-tS}|^2\vo E_h=e^{-2t\re\tr S}\vo E_h,
\]  
which means $\re\tr S\ge 0$ and $\re\tr T/n\le\lambda$. This being true for all  $\lambda>M$, (3.2) follows.

We conclude that $l$ is indeed in the convex hull of the set in (3.1), i.e., there is a Borel measure
$\nu$ on $\partial\Omega\cap\partial E_h$, of total mass $1$, such that $\int l_z\,d\nu(z)=l$. ($\nu$ can be
chosen to have finite support.) Thus with $\mu=n\nu$
\[
\int_{\partial\Omega\cap\partial E_h}\re h(Tz,z)\,d\mu(z)=nl(T)=\re\tr T,\qquad T\in L(\bC^n).
\]
This implies the claimed  necessary condition
\begin{equation} 
\int_{\partial\Omega\cap\partial E_h} h(Tz,z)\,d\mu(z)= \tr T
\end{equation}
because a $\bC$--linear form is uniquely determined by its real part.---The reader will notice that the 
proof of necessity is also quite close to John's proof in \cite{J48}.

To prove that (3.3) is sufficient for maximality, suppose that $E_h\in\bE_0$ inscribed in $\Omega$ and a
Borel measure $\mu$ on $\partial\Omega\cap\partial E_h$ satisfy (3.3) for all linear operators $T$ on
$\bC^n$. We need to show that any $E_k\in\bE_0$ inscribed in $\Omega$ has volume $\le\vo E_h$.

The geodesic $\cE:[0,1]\to\bE_0$ between $E_h,E_k$ is of form $\cE(t)=e^{tT}E_h$, where $T$ is a
self adjoint operator on $(\bC^n,h)$. By Theorem 1.3
\[
\cE(t)=\{z\in\bC^n:h(e^{-tT}z,e^{-tT}z)<1\}\subset\Omega,\qquad 0\le t\le 1.
\]
If $z\in\partial E_h$ and $h(Tz,z)>0$, then 
\[
h(e^{-tT}z,e^{-tT}z)=h(z,z)-2t\re h(Tz,z)+O(t^2)<1
\]
for small $t>0$. Therefore such $z$ is in $\cE(t)\subset\Omega$. Put it differently, if $z\in \partial\Omega\cap\partial E_h$
then $h(Tz,z)\le 0$. By (3.3) $\tr T\le 0$ and
\[
\vo E_k=e^{2\tr T}\vo E_h\le\vo E_h
\]
follow.

\section{Uniqueness of maximal ellipsoids} 

We split Theorem 1.2 in two and first prove
\begin{prop} 
Suppose $\Omega\subset\bC^n$ is a bounded pseudoconvex domain, $0\in\Omega$. If $\partial\Omega$
contains no holomorphic disc, then $\Omega$ admits a unique maximal ellipsoid $E\in\bE_0$.
 \end{prop}
 The following auxiliary result will be needed:
 \begin{prop} 
 Suppose $P\subset\bC^2$ is a pseudoconvex open set and $U\subset\bC^2$ is a neighborhood of $(1,0)$
 such that 
 \begin{equation} 
 \{(r,s)\in U:|r|<1\}\subset P.
 \end{equation}
 If $(1,0)\in\partial P$ then all $(1,s)\in\{1\}\times\bC$ sufficiently close to $(1,0)$ are in $\partial P$.
 \end{prop}
 \begin{proof}
Define $d:\overline P\to[0,\infty]$ by
\[
d(r,s)=\inf\{|\rho|:\rho\in\bC, (r+\rho,s)\notin P\},\qquad (r,s)\in\overline P,
\]
the distance from $(r,s)$ to $\partial P$ in the complex direction $(1,0)$. Then $-\log d$ is plurisubharmonic 
on $P$, see e.g. \cite[Chapter I, (7.2) Theorem]{D12}. (4.1) implies that $d(r,s)\ge 1-|r|$ if $(r,s)$ is sufficiently
close to $(1,0)$; furthermore, $d(r,0)\le|1-r|$, since $(1,0)\notin P$. The maximum principle applied to the
subharmonic function $-\log d(r,\cdot)$, $0<r<1$, gives $-\log d(r,s)=-\log(1-r)$ when $(r,s)$ is close to $(1,0)$. Letting
$r\to 1$ we obtain $d(1,s)=0$, i.e., $(1,s)\notin P$ for small $s$. But (4.1) implies $(1,s)\in\overline P$, hence
$(1,s)\in\partial P$.
\end{proof}
\begin{proof}[Proof of Proposition 4.1]
Suppose $E,E'$ are maximal ellipsoids in $\Omega$ and $\cE:[0,1]\to\bE_0$ is the geodesic between the two.
By Theorem 1.3 each $\cE(t)$ is inscribed in $\Omega$. Set $\cE(1/2)=E_h$. Proposition 2.1 implies that
$\cE(t)=e^{(t-1/2)A}E_h$ with some self adjoint operator $A$ on $(\bC^n,h)$. Now $\tr A=0$ because
$\vo\cE(0)=\vo\cE(1)=e^{2\tr A}\vo\cE(0)$; hence all $\cE(t)$ are of the same volume, therefore maximal. In 
particular, $E_h$ is maximal.

Next we show that $Az=0$ if $z\in\partial\Omega\cap\partial E_h$. With such $z$ let
\[
\Phi(r,s)=re^{sA}z\in\bC^n,\qquad (r,s)\in\bC^2;
\]
$P=\Phi^{-1}\Omega$, and $U=\{(r,s)\in\bC^2:|\re s|<1/2\}$. If $|r|<1$ and $|\re s|<1/2$, then
\[
h\big(e^{-(\re s)A}\Phi(r,s), e^{-(\re s)A}\Phi(r,s)\big)=|r|^2h(z,z)<1,
\]
and so $\Phi(r,s)\in\cE(1/2+\re s)\subset\Omega$. Thus $(r,s)\in P$ and (4.1) holds. In particular, 
$(1,0)\in\overline P$.
But $\Phi(1,0)=z\notin\Omega$, whence $(1,0)\notin P$, and so $(1,0)\in\partial P$. The assumptions of 
Proposition 4.2 satisfied, we conclude that $(1,s)\in\partial P$ and $e^{sA}z=\Phi(1,s)\in\partial\Omega$ for 
small $s\in\bC$. Since, however, $\partial\Omega$ contains no holomorphic disc, $e^{sA}z$ must be 
independent of $s$, and $Az=\partial e^{sA}z/\partial s|_{s=0}=0$.

To finish the proof we recall that $E_h$ is maximal. Theorem 1.1 therefore provides a Borel measure $\mu$
on $\partial\Omega\cap\partial E_h$ such that $\int h(Tz,z)\,d\mu(z)=\tr T$ for all linear operators $T$ on
$\bC^n$. In particular,
\[
\tr A^*A=\int_{\partial\Omega\cap\partial E_h}h(A^*Az,z)\,d\mu(z)=0.
\]
Thus $A=0$ and the two maximal ellipsoids $E,E'$ coincide.
\end{proof}

Now we turn to the remaining half of Theorem 1.2:
\begin{prop} 
There is a bounded pseudoconvex domain $\Omega\subset\bC^2$ that admits infinitely many maximal 
ellipsoids $E\in\bE_0$.
\end{prop}
\begin{proof}
We take
\[
\Omega=\{(x,y)\in\bC^2:|xy|<1\text{ and } |x|,|y|<3\},
\]
and show that with any real number $1/2<p<2$
\[
E=\{(x,y)\in\bC^2:(p^2/2)|x|^2+(p^{-2}/2)|y|^2<1\}\in\bE_0
\]
is a maximal ellipsoid in $\Omega$. First, if $(x,y)\in E$, then
\[
 |x|^2/8<(p^2/2)|x|^2<1,\qquad |y|^2/8<(p^{-2}/2)|y|^2<1,
 \]
and by the inequality of arithmetic and geometric means
$|xy|<1$. These estimates show that $E$ is inscribed in $\Omega$. 
Furthermore, the set $C=\{(p^{-1}e^{it},pe^{iu}): 0\le t,u\le 2\pi\}$ is contained in 
$\partial\Omega\cap\partial E$. Let $\mu$ be the measure on $C$ of total mass 2,  invariant under rotations of
the coordinates. Given a linear operator $T(x,y)=(\alpha x+\beta y,\gamma x+\delta y)$ on $\bC^2$, we compute
\begin{align*}
&\int_C\Big(\frac{p^2}2(\alpha x+\beta y)\overline x+\frac{p^{-2}}2(\gamma x+\delta y)\overline y\Big)\,d\mu(x,y) =\\
\frac1{4\pi^2}&\int_0^{2\pi}
\int_0^{2\pi} \big((\alpha e^{it}+p^2\beta e^{iu})e^{-it}+(p^{-2}\gamma e^{it}+\delta e^{iu})e^{-iu}\big)
\,dt\,du =\alpha+\delta.
\end{align*}
By the criterium in Theorem 1.1 this guarantees that all $E$ are maximal.
\end{proof}

\section {A variant} 

In this section we address the problem of maximizing volume over translates of hermitian ellipsoids, inscribed in a
bounded pseudoconvex domain $\Omega\subset\bC^n$. We will derive a necessary condition, like (1.2),
 for the maximum. However, the similarity with Theorems 1.1, 1.2 stops there. The condition is not
sufficient for maximality; nor are maximizers unique, even if $\Omega$ is strongly pseudoconvex.

Denote by $\bE$ the space of all translates of hermitian ellipsoids in $\bC^n$. This space is in bijective 
correspondence with the space of pairs $(h,c)$ of positive definite hermitian forms $h$ and $c\in\bC^n$,
the ellipsoid corresponding to $(h,c)$ being
\[
E_{h,c}=\{z\in\bC^n:h(z-c,z-c)<1\}.
\]
\begin{thm} 
Suppose $\Omega\subset\bC^n$ is a bounded open set. Among $E\in\bE$ contained in $\Omega$ there is
at least one that has maximal volume. If $E_{h,c}$ is a maximizer, then there is a Borel measure $\mu$ on
$\partial\Omega\cap\partial E_{h,c}$ such that
\begin{align*}
&\int_{\partial\Omega\cap\partial E_{h,c}}(z-c)\,d\mu(z)=0\qquad\text{and}\\
&\int_{\partial\Omega\cap\partial E_ {h,c}} h(Tz,z-c)\,d\mu(z)=\tr T\quad\text{for all linear operators }T
\text{ on } \bC^n.
\end{align*}
\end{thm}

In this section we will refer to the maximizer $E_{h,c}$ above as {\sl maximal} (in $\Omega$).
\begin{proof}
The existence of the maximizer is argued in the same way as in section 3. As for the rest, it suffices to prove
when $c=0$, the general case will then follow upon translating $\Omega$. Thus we take
$E_{h,c}=E_h\in\bE_0$.

We let, as in section 3, $L(\bC^n)$ be the space of linear operators on $\bC^n$. Consider the space
$\bL=L(\bC^n)\oplus\bC^n$ as a real vector space, its dual $\bL'$, and linear forms $l_z\in\bL'$ ($z\in\bC^n)$,
$l\in\bL'$,
\[
l_z(T,a)=\re h(Tz+a,z),\qquad l(T,a)=\re\tr T/n,\qquad (T,a)\in L(\bC^n)\oplus \bC^n.
\]
The point again is that $l$ is in the convex hull of the set
\begin{equation} 
\{l_z:z\in\partial\Omega\cap\partial E_h\}\subset\bL'.
\end{equation}
As in section 3, this will follow from
\begin{equation}
\re\tr T/n\le\max\{\re h(Tz+a,z):z\in\partial\Omega\cap\partial E_h\}=M,\qquad (T,a)\in L(\bC^n)\oplus\bC^n.
\end{equation}

Given $T,a$, let $\lambda>M$ and $S=\lambda\Id-T$. If $z\in\partial\Omega\cap\partial E_h$ then 
$\re h(Sz+a,z)>0$. In fact, there are $\delta>0$ and a neighborhood $N\subset\bC^n$ of 
$\partial\Omega\cap\partial E_h$ such that $\re h(Sz+a,z)>\delta$ when $z\in N$. With $t\in\bR$ let
\[
\cE(t)=\{z\in\bC^n:h(e^{tS}z+ta,e^{tS}z+ta)<1\}.
\]
When $t>0$ is small and $z\in N$,
\[
h(e^{tS}z+ta,e^{tS}z+ta)=h(z,z)+2t\re h(Sz+a,z)+O(t^2)>h(z,z).
\]
Therefore $\overline{\cE(t)}\cap N\subset E_h\subset\Omega$. But
$\overline {\cE(0)}\setminus N=\overline{E_h}\setminus N$ is at positive distance to
$\bC^n\setminus\Omega$, hence $\overline {\cE(t)}\setminus N\subset \Omega$; in sum, 
$\cE(t)\subset\Omega$ for small $t>0$. Since $E_h$ was maximal,
\[
\vo E_h\ge\vo \cE(t)=|\det e^{-tS}|^2\vo E_h=e^{-2t\re\tr S}\vo E_h,
\]  
so that $\re\tr S\ge 0$ and $\re\tr T/n\le\lambda$. This being true for all  $\lambda>M$, (5.2) follows.

We conclude that $l$ is indeed in the convex hull of the set in (5.1), i.e., there is a Borel measure
$\nu$ on $\partial\Omega\cap\partial E_h$, of total mass $1$, such that $\int l_z\,d\nu(z)=l$. 
With $\mu=n\nu$,
\[
\int_{\partial\Omega\cap\partial E_h}\re h(Tz+a,z)\,d\mu(z)=nl(T)=\re\tr T,\qquad (T,a)\in L(\bC^n)\cap\bC^n.
\]
Setting $a=0$, respectively $T=0$, we obtain
\[
\int_{\partial\Omega\cap\partial E_h}\re h(Tz,z)\,d\mu(z)=\re\tr T,\qquad
\re h\big(a,\int_{\partial\Omega\partial E_h}z\,d\mu(z)\big)=0.
\]
As in section 3, we can remove ``$\re$'' to conclude
\[
\int_{\partial\Omega\cap\partial E_h}h(Tz,z)\,d\mu(z)=\tr T,
\qquad \int_{\partial\Omega\cap\partial E_h}z\,d\mu(z)=0,
\]
what was to be proved.
\end{proof}

Yet, the conditions in Theorem 5.1 are not sufficient for maximality, even when $n=1$ (when every domain
is pseudoconvex, and translates of hermitian ellipsoids are  circular discs). For example, let
\[
D=\{z\in\bC:|z-2|<2\}, \qquad E=\{z\in\bC:|z|<1\},
\]
and $\Omega=D\cup E$. Thus $E=E_{h,0}$ with $h(z,w)=z\bar w$. Let $\mu$ be the measure supported
on $\{i,-i\}\subset\partial\Omega\cap\partial E$ giving mass $1/2$ to each of $\pm i$. Clearly $\int z\,d\mu(z)=0$.
Any $T$ is multiplication by some $\tau\in\bC$, hence $\int h(Tz,z)\,d\mu(z)=\tau=\tr T$. The conditions in 
Theorem 5.1 are satisfied, but $E$ is not maximal: $D\subset\Omega$ has bigger area.

Before discussing (the lack of) uniqueness of maximal ellipsoids in $\bE$, recall that if $V\subset\bC^n$ is open, 
a function $u:V\to\bR$ of class $C^2$ is strongly plurisubharmonic if at every $z\in V$ the hermitian matrix
$\big(u_{z_j\bar z_k}(z)\big)_{j,k=1}^n$ is positive definite; and a bounded open set $\Omega\subset\bC^n$ is strongly pseudoconvex if every $p\in\partial\Omega$ has a neighborhood $V\subset\bC^n$ with a strongly
plurisubharmonic function $u:V\to\bR$, $du(p)\neq 0$ and $\Omega\cap V=\{z\in V:u(z)<0\}$.
\begin{thm} 
There is a strongly pseudoconvex domain $\Omega\subset\bC^2$ such that among ellipsoids $E\in\bE$
inscribed in $\Omega$ there is more than one of maximal volume.
\end{thm}
\begin{proof}
Let $p=(1,0)\in\bC^2$, and with a positive number $\lambda$
\[
\Omega=\{z=(x,y)\in\bC^2:|z-p||z+p|<\lambda^2\},
\]
a figure that could be called Cassini's ovaloid\footnote{Cassini's oval signifying the locus of points in the
plane the product of whose distances to two fixed points equals a constant.}, and lemniscatoid when
$\lambda=1$. We will see that $\Omega$
fills the bill when $\lambda>1$ is close to $1$.

We start by computing derivatives of $u(z)=|z-p|^2|z+p|^2-\lambda^4$, a defining function of $\Omega$: 
\begin{gather}
u_x=2\bar x\big(|x|^2+|y|^2\big)-2 x,\qquad u_y=2\bar y\big(|x|^2+|y|^2+1\big),\nonumber\\
u_{x\bar x}=2\big(2|x|^2+|y|^2\big),\qquad u_{x\bar y}=2\bar xy,\qquad u_{y\bar y}=2\big(|x|^2+2|y|^2+1\big). 
\end{gather}
Thus the critical points of $u$ are at $(0,0)$, $(\pm1,0)$, with critical values $1-\lambda^4$ and
$-\lambda^4$. Unless $\lambda=1$, $du$ does not vanish on $\partial\Omega$, and (5.3) shows that $u$ 
is strongly plurisubharmonic away from $(0,0)$. Therefore $\Omega$ is strongly pseudoconvex unless $\lambda=1$.

When $\lambda>1$, $\Omega$ is connected. Indeed, if $z=(x,y)\in\Omega$ and $0\le t\le 1$, then
$z_t=(\re x+t\,\im x,ty)$ satisfies $|z_t-p|\le|z-p|$, $|z_t+p|\le|z+p|$, whence $z_t\in\Omega$. Along the paths
$z_t$ we can deform $\Omega$ to its intersection with the $\re x$--axis, which is the connected interval 
$(-\sqrt{\lambda^2+1},\sqrt{\lambda^2+1})$. We conclude $\Omega$ is connected, so 
a strongly pseudoconvex domain.

Since $\Omega$ is symmetric, if $E\in\bE$ is maximal in it, then $-E\subset\Omega$ is also maximal. Therefore
the inscribed maximal ellipsoid can be unique only if $E=-E$, i.e., if $E$ is centered at $0$, or $E\in\bE_0$.
The question is now whether an $E\in\bE_0$ can be maximal. One can estimate the volume of any inscribed
$E\in\bE_0$ as follows. If $z=(x,y)\in\Omega$ then $|z-p|<\lambda$ or $|z+p|<\lambda$. Accordingly,
$|x-1|<\lambda$ or $|x+1|<\lambda$. In either case $|x|<\lambda+1$, and also $|y|<\lambda$. This implies that
$\Omega$ is contained in the ball $\{\zeta\in\bC^2:|\zeta|<2\lambda+1\}$, and so any ellipsoid inscribed 
in $\Omega$ has its half-axes $\le 2\lambda+1$. But since $(0,\sqrt{\lambda^2-1})\in\partial\Omega$, 
one---and because of the hermitian condition, in fact two---half--axes must have length $\le \sqrt{\lambda^2-1}$.
It follows that 
\[
\vo E\le \pi^2(\lambda^2-1)(2\lambda+1)^2/2.
\]
At the same time, if $\lambda>1$, the ball $B=\{z\in\bC^2:|z-p|<1/3\}\in\bE$ is inscribed, since 
$|z-p|<1/3$ implies 
$|z+p|<7/3$ and $|z-p||z+p|<7/9<\lambda^2$. When $\lambda$ is close enough to $1$, 
\(
\vo B=\pi^2/162>\vo E.
\)
Therefore the maximal ellipsoid in $\bE$ is not in $\bE_0$, and is necessarily not unique.
\end{proof}

\cite{A+22} already features a convex $\Omega\subset\bC$ in which the maximal inscribed $E\in\bE$ is not
unique (any rectangle different from a square will do). It is not impossible that this example can also be developed into
a similar, strongly pseudoconvex example in $\bC^2$.


\begin{thebibliography}{A+22}

\bibitem[A+22]{A+22}  Jorge Arocha, Javier Bracho, Luis Montejano, Extremal inscribed and 
circumscribed complex ellipsoids. Beitr. Algebra Geom. 63 (2022) 349--358


\bibitem[B97]{B97} Keith Ball, An elementary introduction to modern convex geometry. Flavors of geometry, 
1-–58, Math. Sci. Res. Inst. Publ., 31, Cambridge Univ. Press, Cambridge, 1997

\bibitem[D12]{D12} Jean--Pierre Demailly, Complex analytic and differential geometry. 
http://www--fourier.ujf--grenoble.fr/$\sim$demailly/manuscripts/agbook.pdf

\bibitem[G67]{G67} Mikhail Gromov, On a geometric hypothesis of Banach. (Russian) Izv. Akad. Nauk SSSR Ser. Mat. 31 (1967) 1105–-1114.

\bibitem[H91]{H91} Lars H\"ormander, An introduction to complex analysis in several variables, 3rd edition. North Holland, 
Amsterdam etc., 1991.

\bibitem[J48]{J48}  Fritz John, Extremum problems with inequalities as subsidiary conditions. Studies and Essays Presented to R. Courant on his 60th Birthday, January 8, 1948, 187–-204, Interscience Publishers, New York, 1948. 

\bibitem[L24]{L24} L\'aszl\'o Lempert, Two variational problems in K\"ahler geometry, arxiv:2405.00869 

\bibitem[M97]{M97} Daowei Ma, Carathéodory extremal maps of ellipsoids. J. Math. Soc. Japan 49 (1997) 723–-739. 

\bibitem[R86]{R86} V. V. Rabotin, The Carathéodory extremal problem in a class of holomorphic mappings of bounded circular domains. (Russian) Sibirsk. Mat. Zh. 27 (1986) 143–149, 199–200.



\end{thebibliography}
\end{document}